\documentclass[12pt]{article}
\usepackage{amsmath,amscd,amsthm,a4,amssymb}

\newtheorem{theorem}{Theorem}[section]
\newtheorem{lemma}[theorem]{Lemma}
\newtheorem{corollary}[theorem]{Corollary}
\newtheorem{definition}[theorem]{Definition}
\newtheorem{remark}[theorem]{Remark}

\numberwithin{equation}{section}

\usepackage[]{graphicx}

\def\Rn{{\mathbb{R}^n}}

\def\a{\alpha}
\def\b{\beta}
\def\i{\infty}

\def\L1loc{L_{\Phi}^{\rm loc}(\Rn)}

\def\dual{\,^{^{\complement}}\!}

\usepackage{amsmath,amscd,amsthm,amssymb}
\usepackage{color}

\begin{document}

\begin{center}
\Large Fractional maximal function and its commutators on Orlicz spaces
\end{center}

\

\centerline{\large Vagif S. Guliyev$^{a,b,c}$, Fatih Deringoz$^{a,}$\footnote{
Corresponding author.
\\
The research of V.S. Guliyev  was partially supported by the grant of Presidium of Azerbaijan National Academy of Science 2015 and
by the Ministry of Education and Science of the Russian Federation (the Agreement number No. 02.a03.21.0008).
\\
E-mail adresses: deringoz@hotmail.com (F. Deringoz), vagif@guliyev.com (V.S. Guliyev), sabhasanov@gmail.com (S.G. Hasanov).}, Sabir G. Hasanov$^{d}$ }

\

\centerline{$^{a}$\it Department of Mathematics, Ahi Evran University, 40100 Kirsehir, Turkey}

\centerline{$^{b}$ \it S.M. Nikolskii Institute of Mathematics at RUDN University, 117198 Moscow, Russia}

\centerline{$^{c}$\it Institute of Mathematics and Mechanics of NAS of Azerbaijan, AZ1141 Baku, Azerbaijan}

\centerline{$^{d}$\it Ganja State University, Ganja, Azerbaijan}

\

\begin{abstract}
In this paper, we find necessary and sufficient conditions for the boundedness of fractional maximal operator $M_{\a}$ on Orlicz spaces. As an application of this results we consider the boundedness of fractional maximal commutator $M_{b,\a}$ and nonlinear commutator of fractional maximal operator $[b,M_{\a}]$ on Orlicz spaces, when $b$ belongs to the Lipschitz space, by which some new characterizations of the Lipschitz spaces are given.
\end{abstract}

\



\

\section{Introduction}

Norm inequalities for several classical operators of harmonic analysis have been widely studied in the context of Orlicz spaces. It is well known that many of such operators fail to have continuity properties when they act between certain Lebesgue spaces and, in some situations, the Orlicz spaces appear as adequate substitutes. For example,
the Hardy-Littlewood maximal operator is bounded on $L^p$ for $1 < p < \infty $, but not on $L^1$,  but using Orlicz spaces, we can investigate the boundedness of the maximal operator near $p = 1$, see \cite{Kita1,Cianchi1999,GaRa,GaRaSawTan} for  more precise statements.

Let $T$ be the classical singular integral operator, the commutator $[b, T]$ generated by $T$ and a suitable function $b$ is given by
\begin{equation}\label{DefcomSIO}
  [b, T] f=b T(f)-T(bf).
\end{equation}

A well known result due to Coifman, Rochberg and Weiss \cite{CRW} (see e.g. \cite{S.Janson}) states that
$b \in BMO(\Rn)$ if and only if the commutator $[b, T]$ is bounded on $L^p(\Rn)$ for $1 < p < \infty$. In 1978,
Janson \cite{S.Janson} gave some characterizations of the Lipschitz space $\dot{\Lambda}_{\beta}(\Rn)$ (see Definition \ref{deflip} below)
via commutator $[b, T]$ and proved that $b \in \dot{\Lambda}_{\beta}(\Rn)(0 < \beta < 1)$ if and only if $[b, T]$ is bounded
from $L^p(\Rn)$ to $L^q(\Rn)$ where $1 < p < n/\beta$ and $1/p - 1/q = \beta/n$ (see also Paluszy\'{n}ski \cite{Palu}).

Let $0<\a<n$. The fractional maximal operator $M_{\a}$ is given by
\begin{equation*}
M_{\a} f(x)=\sup_{B\ni x}|B|^{-1+ \frac{\a}{n}}\int _{B} |f(y)|dy
\end{equation*}
and the fractional maximal commutator of $M_{\a}$ with a locally integrable function $b$ is defined by
$$
M_{b,\a}f(x)=\sup_{B\ni x}|B|^{-1+ \frac{\a}{n}}\int _{B}|b(x)-b(y)||f(y)|dy,
$$
where the supremum is taken over all balls $B \subset \Rn$ containing $x$. If $\a=0$, then $M \equiv M_{0}$ is the Hardy-Littlewood maximal operator and $M_b \equiv M_{b,0}$ is the maximal commutator of $M$.

On the other hand, similar to \eqref{DefcomSIO}, we can define the (nonlinear) commutator of the fractional maximal operator $M_{\a}$ with a locally integrable function $b$ by
$$
[b, M_{\a}](f)(x) f=b(x) M_{\a}(f)(x)-M_{\a}(bf)(x).
$$
For more details about the operators $M_{b,\a}$ and $[b,M_{\a}]$, where $0\le\a<n$, we refer to \cite{AgcGogKMus,ZhLWu} and references therein.

Our main aim is to characterize the functions involved in the boundedness on Orlicz spaces of the fractional maximal operator $M_{\a}$. Actually, such a characterization was done in \cite[Theorem 1]{Cianchi1999}. But our technique of the proof and characterization different from the ones in \cite{Cianchi1999}. As an application of this result
we consider the boundedness of $M_{b,\a}$ and $[b,M_{\a}]$ on Orlicz spaces when $b$ belongs to the Lipschitz space, by which some new characterizations of the Lipschitz spaces are given.

Throughout the whole paper, the notation $A \lesssim B$ means that there exists a constant
$C > 0$ such that $A \le C B$, where $C$ is independent of appropriate quantities. If $C_1 B \le A \le C_2 B$ for some positive constants $C_1$ and $C_2$, we shall write $A\approx B$.

\

\section{Preliminaries}
Before we proceed with the proofs of the main results, we shall introduce some preliminary deﬁnitions and properties concerning Orlicz spaces.

\begin{definition}\label{def2} A function $\Phi : [0,\infty) \rightarrow [0,\infty]$ is called a Young function if $\Phi$ is convex, left-continuous, $\lim\limits_{r\rightarrow +0} \Phi(r) = \Phi(0) = 0$ and $\lim\limits_{r\rightarrow \infty} \Phi(r) = \infty$.
\end{definition}
From the convexity and $\Phi(0) = 0$ it follows that any Young function is increasing.
%If there exists $s \in  (0,\infty )$ such that $\Phi(s) = \infty $,
%then $\Phi(r) = \infty $ for $r \geq s$.
The set of  Young  functions such that
\begin{equation*}
0<\Phi(r)<\infty \qquad \text{for} \qquad 0<r<\infty
\end{equation*}
will be denoted by  $\mathcal{Y}.$
If $\Phi \in  \mathcal{Y}$, then $\Phi$ is absolutely continuous on every closed interval in $[0,\infty )$
and bijective from $[0,\infty )$ to itself.

For a Young function $\Phi$ and  $0 \leq s \leq \infty $, let
$$\Phi^{-1}(s)=\inf\{r\geq 0: \Phi(r)>s\}.$$
If $\Phi \in  \mathcal{Y}$, then $\Phi^{-1}$ is the usual inverse function of $\Phi$.
%We note that
%$$
%\Phi(\Phi^{-1}(r))\leq r \leq \Phi^{-1}(\Phi(r)) \quad \text{ for } 0\leq r<\infty .
%$$
It is well known that
\begin{equation}\label{2.3}
r\leq \Phi^{-1}(r)\widetilde{\Phi}^{-1}(r)\leq 2r, \qquad  r\geq 0,
\end{equation}
where $\widetilde{\Phi}(r)$ is defined by
\begin{equation*}
\widetilde{\Phi}(r)=\left\{
\begin{array}{ccc}
\sup\{rs-\Phi(s): s\in  [0,\infty )\}
& , & r\in  [0,\infty ) \\
\infty &,& r=\infty .
\end{array}
\right.
\end{equation*}

A Young function $\Phi$ is said to satisfy the
 $\Delta_2$-condition, denoted also as   $\Phi \in  \Delta_2$, if
$$
\Phi(2r)\le C\Phi(r), \qquad r\geq 0
$$
for some $C\geq 2$. If $\Phi \in  \Delta_2$, then $\Phi \in  \mathcal{Y}$. A Young function $\Phi$ is said to satisfy the $\nabla_2$-condition, denoted also by  $\Phi \in  \nabla_2$, if
$$
\Phi(r)\leq \frac{1}{2C}\Phi(Cr),\qquad r\geq 0
$$
for some $C>1$. We can verify the following examples: The function $\Phi(r) = r$ satisfies the $\Delta_2$-condition but does not satisfy the $\nabla_2$-condition.
If $1 < p < \infty$, then $\Phi(r) = r^p$ satisfies both the conditions. The function $\Phi(r) = e^r - r - 1$ satisfies the
$\nabla_2$-condition but does not satisfy the $\Delta_2$-condition.

\begin{definition} (Orlicz Space). For a Young function $\Phi$, the set
$$L^{\Phi}(\Rn)=\left\{f\in  L^1_{\rm loc}(\Rn): \int _{\Rn}\Phi(k|f(x)|)dx<\infty
 \text{ for some $k>0$  }\right\}$$
is called Orlicz space. If $\Phi(r)=r^{p},\, 1\le p<\infty $, then $L^{\Phi}(\Rn)=L^{p}(\Rn)$. If $\Phi(r)=0,\,(0\le r\le 1)$ and $\Phi(r)=\infty ,\,(r> 1)$, then $L^{\Phi}(\Rn)=L^\infty (\Rn)$. The  space $L^{\Phi}_{\rm loc}(\Rn)$ is defined as the set of all functions $f$ such that  $f\chi_{_B}\in  L^{\Phi}(\Rn)$ for all balls $B \subset \Rn$.
\end{definition}
$L^{\Phi}(\Rn)$ is a Banach space with respect to the norm
$$\|f\|_{L^{\Phi}}=\inf\left\{\lambda>0:\int _{\Rn}\Phi\Big(\frac{|f(x)|}{\lambda}\Big)dx\leq 1\right\}.$$

For a measurable set $\Omega\subset \mathbb{R}^{n}$, a measurable function $f$ and $t>0$, let
$
m(\Omega,\ f,\ t)=|\{x\in\Omega:|f(x)|>t\}|.
$
In the case $\Omega=\mathbb{R}^{n}$, we shortly denote it by $m(f,\ t)$.
\begin{definition} The weak Orlicz space
$$
WL^{\Phi}(\mathbb{R}^{n})=\{f\in L^{1}_{\rm loc}(\mathbb{R}^{n}):\Vert f\Vert_{WL^{\Phi}}<\infty\}
$$
is defined by the norm
$$
\Vert f\Vert_{WL^{\Phi}}=\inf\Big\{\lambda>0\ :\ \sup_{t>0}\Phi(t)m\Big(\frac{f}{\lambda},\ t\Big)\ \leq 1\Big\}.
$$
\end{definition}
We note that $\Vert f\Vert_{WL^{\Phi}}\leq \Vert f\Vert_{L^{\Phi}}$,
$$
\sup_{t>0}\Phi(t)m(\Omega,\ f,\ t)=\sup_{t>0}t\,m(\Omega,\ f,\ \Phi^{-1}(t))= \sup_{t>0}t\,m(\Omega,\ \Phi(|f|),\ t)
$$
and
\begin{equation}\label{orlpr}
\int _{\Omega}\Phi\Big(\frac{|f(x)|}{\|f\|_{L^{\Phi}(\Omega)}}\Big)dx\leq 1,\qquad \sup_{t>0}\Phi(t)m\Big(\Omega,\ \frac{f}{\|f\|_{WL^{\Phi}(\Omega)}},\ t\Big)\leq 1,
\end{equation}
where $\|f\|_{L^{\Phi}(\Omega)}=\|f\chi_{_\Omega}\|_{L^{\Phi}}$ and $\|f\|_{WL^{\Phi}(\Omega)}=\|f\chi_{_\Omega}\|_{WL^{\Phi}}$.

The following analogue of the H\"older's inequality is well known (see, for example, \cite{RaoRen}).
\begin{theorem}\label{HolderOr}
Let $\Omega\subset\Rn$ be a measurable set and functions $f$ and $g$ measurable on $\Omega$. For a Young function $\Phi$ and its complementary function  $\widetilde{\Phi}$,
the following inequality is valid
$$\int_{\Omega}|f(x)g(x)|dx \leq 2 \|f\|_{L^{\Phi}(\Omega)} \|g\|_{L^{\widetilde{\Phi}}(\Omega)}.$$
\end{theorem}

By elementary calculations we have the following property.
\begin{lemma}\label{charorlc}
Let $\Phi$ be a Young function and $B$ be a set in $\mathbb{R}^n$ with finite Lebesgue measure. Then
\begin{equation*}
\|\chi_{_B}\|_{L^{\Phi}} = \|\chi_{_B}\|_{WL^{\Phi}}=\frac{1}{\Phi^{-1}\left(|B|^{-1}\right)}.
\end{equation*}
\end{lemma}
By Theorem \ref{HolderOr}, Lemma \ref{charorlc} and \eqref{2.3} we get the following estimate.
\begin{lemma}\label{lemHold} %\cite{DerGulSam}
For a Young function $\Phi$ and $B=B(x,r)$, the following inequality is valid:
$$\int_{B}|f(y)|dy \leq 2 |B| \Phi^{-1}\left(|B|^{-1}\right) \|f\|_{L^{\Phi}(B)}.$$
\end{lemma}

\

\section{\large{The boundedness of fractional maximal operator}}

In this section, we shall give a necessary and sufficient condition for the boundedness of $M_{\a}$ on Orlicz spaces and weak Orlicz spaces. We begin with the boundedness of the maximal operator on Orlicz spaces.
\begin{theorem}\label{Maxorl}\cite{KokKrbec}
Let $\Phi$ be a Young function.

(i)The operator $M$ is bounded from $L^{\Phi}(\Rn)$ to  $WL^{\Phi}(\Rn)$, and the inequality
\begin{equation}\label{MbdninqW}
\|M f\|_{WL^{\Phi}}\leq C_0\|f\|_{L^{\Phi}}
\end{equation}
holds with constant $C_0$ independent of $f$.

(ii) The operator $M$ is bounded on $L^{\Phi}(\Rn)$, and the inequality
\begin{equation}\label{MbdninqS}
\|M f\|_{L^{\Phi}}\leq C_0\|f\|_{L^{\Phi}}
\end{equation}
holds with constant $C_0$ independent of $f$ if and only if $\Phi\in\nabla_2$.
\end{theorem}

We recall that, for functions  $\Phi$ and $\Psi$  from $[0,\infty )$
into $[0,\infty ]$, the function $\Psi$ is said to dominate $\Phi$ globally if there exists
a positive constant $c$  such that $\Phi(s)\le\Psi(cs)$ for all $s\geq0$.

In the theorem below we also use the notation
\begin{equation}\label{chi2.3}
\widetilde{\Psi_P}(s)=\int _{0}^{s}r^{P^{\prime}-1}(\mathcal{B}_{P}^{-1}(r^{P^{\prime}}))^{P^{\prime}}dr,
\end{equation}
where  $1<P\le \infty $ and $\widetilde{\Psi_{P}}(s)$ is the Young conjugate funtion to $\Psi_P(s)$, where $\mathcal{B}_{P}^{-1}(s)$ is inverses to
$$
\mathcal{B}_{P}(s)= \int _{0}^{s}\frac{\Psi(t)}{t^{1+P^{\prime}}}dt.
$$

In \cite{Cianchi1999}, Cianchi found the necessary and sufficient conditions for the boundedness of $M_{\a}$ on Orlicz spaces.
\begin{theorem}\label{CianchiFRMax} Let $0< \a<n$.

(i) $M_{\a}$ is bounded from $L^{\Phi}(\Rn)$ to $WL^{\Psi}(\Rn)$ if and only if
\begin{equation}\label{condweakFrMCia}
\text{$\Phi$ dominates globally the function $Q$,}
\end{equation}
whose inverse is given by
$$
Q^{-1}(r)=r^{\a/n}\Psi^{-1}(r).
$$

(ii) $M_{\a}$ is bounded from $L^{\Phi}(\Rn)$ to $L^{\Psi}(\Rn)$ if and only if
\begin{equation}\label{condstrFrMCia}
\int_{0}^{1} \frac{\Psi(t)}{t^{1+n/(n-\a)}}dt<\i \text{  and $\Phi$ dominates globally the function $\Psi_{n/\a}$}.
\end{equation}
Here, $\Psi_{n/\a}$ is the Young function defined as in \eqref{chi2.3}.
\end{theorem}

In order to prove our main theorem, we also need the following lemma.
\begin{lemma}\label{estFrMax}
If $B_0:=B(x_0,r_0)$, then $|B_0|^{\frac{\a}{n}}\leq  M_{\a} \chi_{B_0}(x)$ for every $x\in B_0$.
\end{lemma}
\begin{proof}
For $x\in B_0$, we get
\begin{align*}
M_{\a} \chi_{B_0}(x)= \sup\limits_{B\ni x}|B|^{-1+ \frac{\a}{n}}|B \cap B_0|\geq  |B_0|^{-1+ \frac{\a}{n}}|B_0 \cap B_0|=  |B_0|^{\frac{\a}{n}}.
\end{align*}
\end{proof}

The following result completely characterizes the boundedness of $M_{\a}$ on Orlicz spaces.
\begin{theorem}\label{AdamsFrMaxCharOrl}
Let $0< \a<n$, $\Phi, \Psi$ be Young functions and $\Phi\in\mathcal{Y}$. The condition
\begin{equation}\label{adRieszCharOrl2}
r^{-\frac{\alpha}{n}}\Phi^{-1}\big(r\big)\le C \Psi^{-1}\big(r\big)
\end{equation}
for all $r>0$, where $C>0$ does not depend on $r$, is necessary and sufficient for the boundedness of $M_{\a}$ from $L^{\Phi}(\Rn)$ to $WL^{\Psi}(\Rn)$. Moreover, if $\Phi\in\nabla_2,$ the condition \eqref{adRieszCharOrl2} is necessary and sufficient for the boundedness of $M_{\a}$ from $L^{\Phi}(\Rn)$ to $L^{\Psi}(\Rn)$.
\end{theorem}

\begin{proof}
For a ball $B=B(x,r)$, let $f_1=f\chi _{B(x,2r)}$, $f_2=f-f_1$ and $y$ be an arbitrary point in $B$. If $B(y,t)\cap {\dual}(B(x,2r))\neq\emptyset,$ then $t>r$. Indeed, if $z\in B(y,t)\cap  {\dual} (B(x,2r)),$
then $t > |y-z| \geq |x-z|-|x-y|>2r-r=r$.

On the other hand, $B(y,t)\cap {\dual} (B(x,2r))\subset B(x,2t)$. Indeed, if  $z\in B(y,t)\cap {\dual} (B(x,2r))$, then
we get $|x-z|\leq |y-z|+|x-y|<t+r<2t$.

Hence by Lemma \ref{lemHold}
\begin{equation*}
\begin{split}
M_{\a} f_2(y) & \lesssim \sup_{t>0}\frac{1}{|B(y,t)|^{1-\frac{\a}{n}}} \int_{B(y,t)\cap {{\dual}(B(x,2r))}}|f(z)|d z
%\\
%& \lesssim \, \sup_{t>r}\frac{1}{|B(x,2t)|^{1-\frac{\a}{n}}} \int_{B(x,2t)}|f(z)|d z
\\
&\lesssim  \, \sup_{t>2r} \frac{1}{|B(x,t)|^{1-\frac{\a}{n}}} \int_{B(x,t)}|f(z)| d z
\\
%&\lesssim \sup_{r<t<\infty} t^{\a} \, \Phi^{-1}(|B(x,t)|^{-1}) \, \|f\|_{L^{\Phi}(B(x,t))}\\
&\lesssim \|f\|_{L^{\Phi}} \sup_{r<t<\infty} t^{\a} \, \Phi^{-1}(|B(x,t)|^{-1}).
\end{split}
\end{equation*}
Consequently from Hedberg's trick, see \cite{Hedberg}, and the last inequality, we have
\begin{equation*}
M_\a f(y) \lesssim
r^\a Mf(y)+\|f\|_{L^{\Phi}} \sup_{r<t<\infty} t^{\a} \, \Phi^{-1}(t^{-n}).
\end{equation*}

Thus, by \eqref{adRieszCharOrl2} we obtain
\begin{align*}
|M_{\a} f(y)|  \lesssim  Mf(y)\frac{\Psi^{-1}(r^{-n})}{\Phi^{-1}(r^{-n})} + \|f\|_{L^{\Phi}} \, \Psi^{-1}(r^{-n}).
\end{align*}
Choose $r>0$ so that $\Phi^{-1}(r^{-n})=\frac{Mf(y)}{C_0\|f\|_{L^{\Phi}}}$. Then
$$
\frac{\Psi^{-1}(r^{-n})}{\Phi^{-1}(r^{-n})}=\frac{(\Psi^{-1}\circ\Phi)(\frac{Mf(y)}{C_0\|f\|_{L^{\Phi}}})}{\frac{Mf(y)}{C_0\|f\|_{L^{\Phi}}}}.
$$
Therefore, we get for all $y\in B$
$$
|M_{\a} f(y)|  \leq C_1 \|f\|_{L^{\Phi}} (\Psi^{-1}\circ\Phi) \Big(\frac{Mf(y)}{C_0\|f\|_{L^{\Phi}}}\Big).
$$

Let $C_0$ be as in \eqref{MbdninqW}. Then by Theorem \ref{Maxorl}, we have
\begin{align*}
&\sup_{r>0}\Psi(r)\, m\Big(B, \frac{|M_{\a} f(y)|}{C_1\|f\|_{L^{\Phi}}},r\Big)=\sup_{r>0}r\, m\Big(B, \Psi\Big(\frac{|M_{\a} f(y)|}{C_1\|f\|_{L^{\Phi}}}\Big),r\Big)
\\
\leq &\sup_{r>0}r\, m\Big(B, \Phi\Big(\frac{M f(y)}{C_0\|f\|_{L^{\Phi}}}\Big),r\Big)\leq\sup_{r>0}\Phi(r)\, m\Big(\frac{M f(z)}{\|Mf\|_{WL^{\Phi}}},r\Big)\leq 1,
\end{align*}
i.e.
\begin{equation}\label{gfhajyufr}
\|M_{\a}f\|_{WL^{\Psi}(B)}\lesssim \|f\|_{L^{\Phi}}.
\end{equation}
By taking supremum over $B$ in \eqref{gfhajyufr}, we get
$$
\|M_{\a}f\|_{WL^{\Psi}}\lesssim \|f\|_{L^{\Phi}},
$$
since the constants in \eqref{gfhajyufr} don't depend on $x$ and $r$.

Let $C_0$ be as in \eqref{MbdninqS}. Since $\Phi\in\nabla_2$, by Theorem \ref{Maxorl}, we have
$$
\int_{B}\Psi\left(\frac{|M_{\a} f(y)|}{C_1\|f\|_{L^{\Phi}}}\right)dy\leq \int_{B}\Phi\left(\frac{M f(y)}{C_0\|f\|_{L^{\Phi}}}\right)dy\leq \int_{\Rn}\Phi\left(\frac{M f(z)}{\|Mf\|_{L^{\Phi}}}\right)dz\leq 1,
$$
i.e.
\begin{equation}\label{gfhaj}
\|M_{\a}f\|_{L^{\Psi}(B)}\lesssim \|f\|_{L^{\Phi}}.
\end{equation}
By taking supremum over $B$ in \eqref{gfhaj}, we get
$$
\|M_{\a}f\|_{L^{\Psi}}\lesssim \|f\|_{L^{\Phi}},
$$
since the constants in \eqref{gfhaj} don't depend on $x$ and $r$.

We shall now prove the necessity. Let $B_0=B(x_0,r_0)$ and $x\in B_0$. By Lemma \ref{estFrMax}, we have $r_0^{\alpha}\leq C M_{\a} \chi_{B_0}(x)$.
Therefore, by Lemma \ref{charorlc}, we have
\begin{align*}
r_0^{\alpha}&\lesssim \Psi^{-1}(|B_0|^{-1})\|M_{\a} \chi_{B_0}\|_{WL^{\Psi}(B_0)} \lesssim \Psi^{-1}(|B_0|^{-1})\|M_{\a} \chi_{B_0}\|_{WL^{\Psi}}
\\
&\lesssim \Psi^{-1}(|B_0|^{-1})\|\chi_{B_0}\|_{L^{\Phi}}\lesssim \frac{\Psi^{-1}(r_{0}^{-n})}{\Phi^{-1}(r_{0}^{-n})}
\end{align*}
and
\begin{align*}
r_0^{\alpha}&\lesssim \Psi^{-1}(|B_0|^{-1})\|M_{\a} \chi_{B_0}\|_{L^{\Psi}(B_0)} \lesssim \Psi^{-1}(|B_0|^{-1})\|M_{\a} \chi_{B_0}\|_{L^{\Psi}}
\\
&\lesssim \Psi^{-1}(|B_0|^{-1})\|\chi_{B_0}\|_{L^{\Phi}}\lesssim \frac{\Psi^{-1}(r_{0}^{-n})}{\Phi^{-1}(r_{0}^{-n})}.
\end{align*}
Since this is true for every $r_0>0$, we are done.
\end{proof}

We recover the following well known result by taking $\Phi(t)=t^p$ at Theorem \ref{AdamsFrMaxCharOrl}.
\begin{corollary}
Let $0< \a<n$ and $1\le p\le n/\a$. Then the condition $1/q=1/p-\a/n$ is necessary and sufficient for the boundedness of $M_{\a}$ from
$L^{p}(\Rn)$ to $WL^{q}(\Rn)$ and for $p>1$ from $L^{p}(\Rn)$ to $L^{q}(\Rn)$.
\end{corollary}

From Theorems \ref{CianchiFRMax} and \ref{AdamsFrMaxCharOrl} we have the following corollary.
\begin{corollary} Let $0<\a<n$, $\Phi, \Psi$ be Young functions and $\Phi\in\mathcal{Y}$, then:
	
1) Condition \eqref{condweakFrMCia} holds if and only if condition \eqref{adRieszCharOrl2} holds.
	
2) Moreover if $\Phi\in\nabla_2$, then condition \eqref{condstrFrMCia} holds if and only if \eqref{adRieszCharOrl2} holds.
\end{corollary}

\

\section{Characterization of Lipschitz spaces via commutators}
In this section, as an application of Theorem \ref{AdamsFrMaxCharOrl} we consider the boundedness of $M_{b,\a}$ and $[b,M_{\a}]$ on Orlicz spaces when $b$ belongs to the Lipschitz space, by which some new characterizations of the Lipschitz spaces are given.

\begin{definition}\label{deflip}
Let $0 < \beta < 1$, we say a function $b$ belongs to the Lipschitz space $\dot{\Lambda}_{\beta}(\Rn)$ if there exists a constant $C$ such that for all $x, y \in \Rn$,
$$
|b(x)-b(y)|\le C|x-y|^{\beta}.
$$
The smallest such constant $C$ is called the $\dot{\Lambda}_{\beta}(\Rn)$ norm of $b$ and is denoted by $\|b\|_{\dot{\Lambda}_{\beta}(\Rn)}$.
\end{definition}

To prove the theorems, we need auxiliary results. The first one is the following characterizations
of Lipschitz space, which is due to DeVore and Sharply \cite{DeVSharp}.
\begin{lemma}\label{CharLipSp}
Let $0 < \beta < 1$, we have
$$
\|f\|_{\dot{\Lambda}_{\beta}(\Rn)}\thickapprox\sup_{B}\frac{1}{|B|^{1+\beta/n}}\int_{B}|f(x)-f_{B}|dx,
$$
where $f_{B}=\frac{1}{|B|}\int_{B}f(y)dy$.
\end{lemma}
\begin{lemma}\label{pwslip}
Let $0<\beta<1$, $0\le \a<n$, $0< \a+\b<n$ and $b\in \dot{\Lambda}_{\beta}(\Rn)$, then the following pointwise estimate holds:
$$
M_{b,\a}f(x)\leq C\|b\|_{\dot{\Lambda}_{\beta}(\Rn)}M_{\a+\b}f(x).
$$
\end{lemma}

\begin{proof}
If $b\in \dot{\Lambda}_{\beta}(\Rn)$, then
\begin{align*}
M_{b,\a}(f)(x)&=\sup\limits_{B\ni x}|B|^{-1+ \frac{\a}{n}}\int _{B} |b(x)-b(y)||f(y)|dy\\
&\leq C \|b\|_{\dot{\Lambda}_{\beta}(\Rn)} \sup\limits_{B\ni x}|B|^{-1+ \frac{\a+\b}{n}}\int _{B} |f(y)|dy\\
&= C\|b\|_{\dot{\Lambda}_{\beta}(\Rn)}M_{\a+\b}f(x).
\end{align*}
\end{proof}

\begin{lemma}\label{estFrMaxCom}
If $b\in  L^1_{\rm loc}(\Rn)$ and $B_0:=B(x_0,r_0)$, then
$$
|B_0|^{\frac{\a}{n}}|b(x)-b_{B_0}|\leq M_{b,\a} \chi_{B_0}(x) ~~ \mbox{for every} ~~ x\in B_0.
$$
\end{lemma}
\begin{proof}
%It is well-known that
%\begin{equation}\label{poicomcumaxcom}
%\mathrm{M}_{b,\a}f(x)\leq 2^{n-\a}M_{b,\a}f(x),
%\end{equation}
%where $\mathrm{M}_{b,\a}(f)(x)=\sup\limits_{B\ni x}|B|^{-1+ \frac{\a}{n}}\int _{B} |b(x)-b(y)||f(y)|dy$.
%
%Now let $x\in B_0$. By using \eqref{poicomcumaxcom}, we get
For $x\in B_0$, we get
\begin{align*}
& M_{b,\a} \chi_{B_0}(x)=\sup\limits_{B\ni x}|B|^{-1+ \frac{\a}{n}}\int _{B} |b(x)-b(y)|\chi_{B_0}(y)dy
\\
&=\sup\limits_{B\ni x}|B|^{-1+ \frac{\a}{n}}\int _{B\cap B_0} |b(x)-b(y)|dy
\geq |B_0|^{-1+ \frac{\a}{n}}\int _{B_0\cap B_0} |b(x)-b(y)|dy
\\
&\geq \big| |B_0|^{-1+ \frac{\a}{n}}\int _{B_0} (b(x)-b(y))dy\big|
=|B_0|^{\frac{\a}{n}}|b(x)-b_{B_0}|.
\end{align*}
\end{proof}

The following theorem is valid.
\begin{theorem} \label{CommFrMaxCharOr01}
Let $0<\beta<1$, $0\le \a<n$, $0< \a+\b<n$, $b\in  L^1_{\rm loc}(\Rn)$, $\Phi, \Psi$ be Young functions and $\Phi\in\mathcal{Y}$.

$1.~$ If $\Phi\in \nabla_2$ and the condition
\begin{equation}\label{adFrCharOrl1M}
t^{-\frac{\alpha+\b}{n}}\Phi^{-1}\big(t\big)\le C \Psi^{-1}\big(t\big),
\end{equation}
holds for all $t>0$, where $C>0$ does not depend on $t$, then the condition $b\in \dot{\Lambda}_{\beta}(\Rn)$
is sufficient for the boundedness of $M_{b,\a}$ from $L^{\Phi}(\Rn)$ to $L^{\Psi}(\Rn)$.

$2.~$ If the condition
\begin{equation}\label{adFrCharOrl1Mncs}
\Psi^{-1}(t) \leq C \Phi^{-1}(t)t^{-\frac{\alpha+\b}{n}},
\end{equation}
holds for all $t>0$, where $C>0$ does not depend on $t$, then the condition $b\in \dot{\Lambda}_{\beta}(\Rn)$
is necessary for the boundedness of $M_{b,\a}$ from $L^{\Phi}(\Rn)$ to $L^{\Psi}(\Rn)$.

$3.~$ If $\Phi\in \nabla_2$ and $\Psi^{-1}(t) \thickapprox \Phi^{-1}(t)t^{-\frac{\alpha+\b}{n}}$,
then the condition $b\in \dot{\Lambda}_{\beta}(\Rn)$ is necessary and sufficient for the boundedness of $M_{b,\a}$ from $L^{\Phi}(\Rn)$ to $L^{\Psi}(\Rn)$.
\end{theorem}
\begin{proof}
(1) The first statement of the theorem follows from Theorem \ref{AdamsFrMaxCharOrl} and Lemma \ref{pwslip}.

(2) We shall now prove the second part. Suppose that $\Psi^{-1}(t) \lesssim \Phi^{-1}(t)t^{-(\a+\b)/n}$
and $M_{b,\a}$ is bounded from $L^{\Phi}(\Rn)$ to $L^{\Psi}(\Rn)$. Choose any ball $B$ in $\Rn$, by Lemmas \ref{charorlc} and \ref{lemHold}
\begin{align*}%\label{ujy02}
\frac{1}{|B|^{1+\frac{\b}{n}}} \int_{B}|b(y)-b_{B}|dy & = \frac{1}{|B|^{1+\frac{\a+\b}{n}}}
\int_{B} \Big|\frac{1}{|B|^{1-\frac{\a}{n}}} \int_{B} (b(y)-b(z))dz \Big| dy \notag
\\
& \le \frac{1}{|B|^{1+\frac{\a+\b}{n}}} \int_{B} M_{b,a}\big( \chi_{_B}\big)(y) dy \notag
\\
& \le \frac{2\Psi^{-1}(|B|^{-1})}{|B|^{\frac{\a+\b}{n}}} \, \|M_{b,\a}\big( \chi_{_B}\big)\|_{L^{\Psi}(B)} \notag
\\
& \le \frac{C}{|B|^{\frac{\a+\b}{n}}} \, \frac{\Psi^{-1}(|B|^{-1})}{\Phi^{-1}(|B|^{-1})} \leq C.
\end{align*}
Thus by Lemma \ref{CharLipSp} we get $b\in \dot{\Lambda}_{\beta}(\Rn)$.

(3) The third statement of the theorem follows from the first and second parts of the theorem.
\end{proof}

If we take $\a=0$ at Theorem \ref{CommFrMaxCharOr01}, we have the following result.
\begin{corollary} \label{BBCommFrMaxCharOr01}
Let $0<\beta<1$, $b\in  L^1_{\rm loc}(\Rn)$, $\Phi, \Psi$ be Young functions and $\Phi\in\mathcal{Y}$.

$1.~$ If $\Phi\in \nabla_2$ and the condition $\Phi^{-1}(t)t^{-\beta/n} \lesssim \Psi^{-1}(t)$ holds,
then the condition $b\in \dot{\Lambda}_{\beta}(\Rn)$ is sufficient for the boundedness of $M_{b}$ from $L^{\Phi}(\Rn)$ to $L^{\Psi}(\Rn)$.

$2.~$ If $\Psi^{-1}(t) \lesssim \Phi^{-1}(t)t^{-\beta/n}$, then the condition $b\in \dot{\Lambda}_{\beta}(\Rn)$ is necessary for the boundedness of $M_{b}$ from $L^{\Phi}(\Rn)$ to $L^{\Psi}(\Rn)$.

$3.~$ If $\Phi\in \nabla_2$ and $\Psi^{-1}(t) \thickapprox \Phi^{-1}(t)t^{-\beta/n}$,
then the condition $b\in \dot{\Lambda}_{\beta}(\Rn)$ is necessary and sufficient for the boundedness of $M_{b}$ from $L^{\Phi}(\Rn)$ to $L^{\Psi}(\Rn)$.
\end{corollary}

If we take $\Phi(t)=t^{p}$ and $\Psi(t)=t^{q}$ with $1\le p<\i$ and $1\le q \le\i$ at Theorem \ref{CommFrMaxCharOr01}, we have the following result.
\begin{corollary}\label{ZhLpCo}
Let $0<\beta<1$, $0\le \a<n$, $0< \a+\b<n$, $b\in  L^1_{\rm loc}(\Rn)$, $1< p<q\le \i$ and $\frac{1}{p}-\frac{1}{q}=\frac{\a+\b}{n}$. Then the condition $b\in \dot{\Lambda}_{\beta}(\Rn)$ is necessary and sufficient for the boundedness of $M_{b,\a}$ from $L^{p}(\Rn)$ to $L^{q}(\Rn)$.
\end{corollary}
\begin{remark}
For $\a=0$, Corollary \ref{ZhLpCo} was proved in \cite{ZhArxiv}.
\end{remark}

The following theorem is valid.
\begin{theorem} \label{CommFrMaxCharOr01W}
Let $0<\beta<1$, $0\le \a<n$, $0< \a+\b<n$, $b\in  L^1_{\rm loc}(\Rn)$, $\Phi, \Psi$ be Young functions and $\Phi\in\mathcal{Y}$.

$1.~$ If condition \eqref{adFrCharOrl1M} holds,
then the condition $b\in \dot{\Lambda}_{\beta}(\Rn)$ is sufficient for the boundedness of $M_{b,\a}$ from $L^{\Phi}(\Rn)$ to $WL^{\Psi}(\Rn)$.

$2.~$ If condition \eqref{adFrCharOrl1Mncs} holds and $\frac{t^{1+\varepsilon}}{\Psi(t)}$ is almost decreasing for some $\varepsilon>0$, then the condition $b\in \dot{\Lambda}_{\beta}(\Rn)$ is necessary for the boundedness of $M_{b,\a}$ from $L^{\Phi}(\Rn)$ to $WL^{\Psi}(\Rn)$.

$3.~$ If $\Psi^{-1}(t) \thickapprox \Phi^{-1}(t)t^{-(\a+\b)/n}$ and $\frac{t^{1+\varepsilon}}{\Psi(t)}$ is almost decreasing for some $\varepsilon>0$,
then the condition $b\in \dot{\Lambda}_{\beta}(\Rn)$ is necessary and sufficient for the boundedness of $M_{b,\a}$ from $L^{\Phi}(\Rn)$ to $WL^{\Psi}(\Rn)$.
\end{theorem}
\begin{proof}
(1) The first statement of the theorem follows from Theorem \ref{AdamsFrMaxCharOrl} and Lemma \ref{pwslip}.

(2) For any fixed ball $B_0$ such that $x\in B_0$ by Lemma \ref{estFrMaxCom} we have $|B_0|^{\alpha/n}|b(x)-b_{B_0}|\leq M_{b,\a} \chi_{B_0}(x)$.
This together with the boundedness of $M_{b,\a}$ from $L^{\Phi}(\Rn)$ to $WL^{\Psi}(\Rn)$ and Lemma \ref{charorlc}
\begin{align*}
  |\{x\in B_0: |B_0|^{\alpha/n}|b(x)-b_{B_0}|> \lambda\}| &\leq |\{x\in B_0: M_{b,\a} \chi_{B_0}(x)> \lambda\}|
  \\
  & \leq \frac{1}{\Psi\left(\frac{ \lambda}{C\|\chi_{B_0}\|_{L^{\Phi}}}\right)} = \frac{1}{\Psi\left(\frac{ \lambda \Phi^{-1}(|B_0|^{-1})}{C}\right)}.
\end{align*}

Let $t>0$ be a constant to be determined later, then
\begin{align*}
  \int_{B_0} |b(x)-b_{B_0}|dx & = |B_0|^{-\alpha/n}\int_{0}^{\infty} |\{x\in B_0: |b(x)-b_{B_0}|>|B_0|^{-\alpha/n}\lambda\}| d\lambda\\
  & = |B_0|^{-\alpha/n}\int_{0}^{t} \{x\in B_0: |b(x)-b_{B_0}|>|B_0|^{-\alpha/n}\lambda\}| d\lambda\\
  & ~~~~+ |B_0|^{-\alpha/n}\int_{t}^{\infty} |\{x\in B_0: |b(x)-b_{B_0}|>|B_0|^{-\alpha/n}\lambda\}| d\lambda \\
  & \le t|B_0|^{1-\alpha/n}+ |B_0|^{-\alpha/n}\int_{t}^{\infty} \frac{1}{\Psi\left(\frac{ \lambda \Phi^{-1}(|B_0|^{-1})}{C}\right)} d\lambda\\
  & \lesssim  t|B_0|^{1-\alpha/n}+\frac{|B_0|^{-\alpha/n}t}{\Psi\left(\frac{ t \Phi^{-1}(|B_0|^{-1})}{C}\right)},
\end{align*}
where we use almost decreasingness of $\frac{t^{1+\varepsilon}}{\Psi(t)}$ in the last step.

Set $t=C|B_0|^{\frac{\alpha+\beta}{n}}$ in the above estimate, we have
$$
\int_{B_0} |b(x)-b_{B_0}|dx\lesssim |B_0|^{1+\b/n}.
$$
Thus by Lemma \ref{CharLipSp} we get $b\in \dot{\Lambda}_{\beta}(\Rn)$ since $B_0$ is an arbitrary ball in $\Rn$.

(3) The third statement of the theorem follows from the first and second parts of the theorem.
\end{proof}

If we take $\a=0$ at Theorem \ref{CommFrMaxCharOr01W}, we have the following result.

\begin{corollary} \label{BBCommFrMaxCharOr01W}
Let $0<\beta<1$, $b\in  L^1_{\rm loc}(\Rn)$, $\Phi, \Psi$ be Young functions and $\Phi\in\mathcal{Y}$.

$1.~$ If the condition $\Phi^{-1}(t)t^{-\beta/n} \lesssim \Psi^{-1}(t)$ holds,
then the condition $b\in \dot{\Lambda}_{\beta}(\Rn)$ is sufficient for the boundedness of $M_{b}$ from $L^{\Phi}(\Rn)$ to $WL^{\Psi}(\Rn)$.

$2.~$ If $\Psi^{-1}(t) \lesssim \Phi^{-1}(t)t^{-\beta/n}$ and $\frac{t^{1+\varepsilon}}{\Psi(t)}$ is almost decreasing for some $\varepsilon>0$, then the condition $b\in \dot{\Lambda}_{\beta}(\Rn)$ is necessary for the boundedness of $M_{b}$ from $L^{\Phi}(\Rn)$ to $WL^{\Psi}(\Rn)$.

$3.~$ If $\Psi^{-1}(t) \thickapprox \Phi^{-1}(t)t^{-\beta/n}$ and $\frac{t^{1+\varepsilon}}{\Psi(t)}$ is almost decreasing for some $\varepsilon>0$,
then the condition $b\in \dot{\Lambda}_{\beta}(\Rn)$ is necessary and sufficient for the boundedness of $M_{b}$ from $L^{\Phi}(\Rn)$ to $WL^{\Psi}(\Rn)$.
\end{corollary}

If we take $\Phi(t)=t^{p}$ and $\Psi(t)=t^{q}$ with $1\le p<\i$ and $1\le q \le\i$ at Theorem \ref{CommFrMaxCharOr01W}, we have the following result.
\begin{corollary}\label{ZhLpCoW}
Let $0<\beta<1$, $0\le \a<n$, $0< \a+\b<n$, $b\in  L^1_{\rm loc}(\Rn)$, $1\le p<q\le\i$ and $\frac{1}{p}-\frac{1}{q}=\frac{\a+\b}{n}$. Then the condition $b\in \dot{\Lambda}_{\beta}(\Rn)$ is necessary and sufficient for the boundedness of $M_{b,\a}$ from $L^{p}(\Rn)$ to $WL^{q}(\Rn)$.
\end{corollary}
\begin{remark}
For $\a=0$, Corollary \ref{ZhLpCoW} was proved in \cite{ZhArxiv}.
\end{remark}

To state our results, we recall the definition of the maximal operator with respect to a ball. For a fixed ball $B_0$, the fractional
maximal function with respect to $B_0$ of a function $f$ is given by
$$
M_{\alpha, B_0}(f)(x)=\sup_{B_{0}\supseteq B\ni x}\frac{1}{|B_0|^{1-\frac{\a}{n}}}\int_{B}|f(y)|dy,\qquad 0\le \a<n,
$$
where the supremum is taken over all the balls $B$ with $B\subseteq B_0$ and $x\in B$.

\begin{theorem}\label{dnndk}
Let $0<\beta<1$, $0\le \a<n$, $0< \a+\b<n$ and $b$ be a locally integrable non-negative function. Suppose that $\Phi, \Psi$ be Young functions, $\Phi\in\mathcal{Y}\cap\nabla_2$ and $\Psi^{-1}(t) \thickapprox \Phi^{-1}(t)t^{-\frac{\a+\b}{n}}$. Then the following statements are equivalent:

$1.~$ $b\in \dot{\Lambda}_{\beta}(\Rn)$.

$2.~$ $[b,M_\alpha]$ is bounded from $L^{\Phi}(\Rn)$ to $L^{\Psi}(\Rn)$.

$3.~$ There exists a constant $C > 0$ such that
\begin{equation}\label{maxbalfx}
\sup_{B} |B|^{-\beta/n}\Psi^{-1}\big(|B|^{-1}\big)\|b(\cdot)-|B|^{-\a/n}M_{\a,B}(b)(\cdot)\|_{L^{\Psi}(B)}\le C.
\end{equation}
\end{theorem}

\begin{proof}
$(1)\Rightarrow (2)$: The following estimate was proved in \cite{ZhWu}. Let $b$ be any non-negative locally integrable function. Then
\begin{equation}\label{3.1gogmus}
|[b,M_\alpha](f)(x)|\leq M_{b,\a}(f)(x),\qquad x\in\Rn
\end{equation}
holds for all $f \in L^1_{\rm loc}(\Rn)$.

It follows from \eqref{3.1gogmus} and Theorem \ref{CommFrMaxCharOr01} that $[b,M_\a]$ is bounded from $L^{\Phi}(\Rn)$ to $L^{\Psi}(\Rn)$ since $b\in \dot{\Lambda}_{\beta}(\Rn)$.

$(2)\Rightarrow (3)$: For any fixed ball $B \subset \Rn$ and all $x \in B$, we have (see (2.4) in \cite{ZhLWu}).
$$
M_{\a}(\chi_{B})(x)=|B|^{\a/n}\qquad \text{and} \qquad M_{\a}(b\chi_{B})(x)=M_{\a,B}(b)(x).
$$
Then,
\begin{align}\label{4.2arxiv}
  &|B|^{-\beta/n}\Psi^{-1}\big(|B|^{-1}\big)\|b(\cdot)-|B|^{-\a/n}M_{\a,B}(b)(\cdot)\|_{L^{\Psi}(B)} \notag  \\
  & = |B|^{-\frac{\a+\b}{n}}\Psi^{-1}\big(|B|^{-1}\big)\|b(\cdot)M_{\a}(\chi_{B})(\cdot)-M_{\a}(b\chi_{B})(\cdot)\|_{L^{\Psi}(B)} \notag \\
  & = |B|^{-\frac{\a+\b}{n}}\Psi^{-1}\big(|B|^{-1}\big)\|[b,M_{\a}](\chi_{B})\|_{L^{\Psi}(B)} \\
  & \leq C |B|^{-\frac{\a+\b}{n}}\Psi^{-1}\big(|B|^{-1}\big)\|\chi_{B}\|_{L^{\Phi}} \notag \\
  & \leq C \notag
\end{align}
which implies (3) since the ball $B \subset \Rn$ is arbitrary.

$(3)\Rightarrow (1)$: From \cite{ZhWu} we have,
$$
\frac{1}{|B|^{1+\frac{\beta}{n}}} \int_{B}|b(x)-b_{B}|dx\leq \frac{2}{|B|^{1+\frac{\beta}{n}}}\int_{B}|b(x)-|B|^{-\a/n}M_{\a,B}(b)(x)|dx.
$$
it follows from Lemma \ref{lemHold} and \eqref{maxbalfx} that
$$
\frac{1}{|B|^{1+\frac{\beta}{n}}} \int_{B}|b(x)-b_{B}|dx\leq \frac{4}{|B|^{\frac{\beta}{n}}}\Psi^{-1}\big(|B|^{-1}\big)\|b(\cdot)-|B|^{-\a/n}M_{\a,B}(b)(\cdot)\|_{L^{\Psi}(B)}\le C.
$$
Thus by Lemma \ref{CharLipSp} we get $b\in \dot{\Lambda}_{\beta}(\Rn)$.
\end{proof}

If we take $\a=0$ at Theorem \ref{dnndk}, we have the following result.
\begin{corollary}
Let $0<\beta<1$ and $b$ be a locally integrable non-negative function. Suppose that $\Phi, \Psi$ be Young functions, $\Phi\in\mathcal{Y}\cap\nabla_2$ and $\Psi^{-1}(t) \thickapprox \Phi^{-1}(t)t^{-\frac{\b}{n}}$. Then the following statements are equivalent:

$1.~$ $b\in \dot{\Lambda}_{\beta}(\Rn)$.

$2.~$ $[b,M]$ is bounded from $L^{\Phi}(\Rn)$ to $L^{\Psi}(\Rn)$.

$3.~$ There exists a constant $C > 0$ such that
\begin{equation*}
\sup_{B} |B|^{-\beta/n}\Psi^{-1}\big(|B|^{-1}\big)\|b(\cdot)-M_{B}(b)(\cdot)\|_{L^{\Psi}(B)}\le C.
\end{equation*}
\end{corollary}

If we take $\Phi(t)=t^{p}$ and $\Psi(t)=t^{q}$ with $1\le p<\i$ and $1\le q \le\i$ at Theorem \ref{dnndk}, we have the following result.
\begin{corollary}\label{ZhLpCoxt}
Let $0<\beta<1$, $0\le \a<n$, $0< \a+\b<n$, $b\in  L^1_{\rm loc}(\Rn)$, $b$ be a locally integrable non-negative function, $1< p<q\le\i$ and $\frac{1}{p}-\frac{1}{q}=\frac{\a+\b}{n}$.
Then the following statements are equivalent:

$1.~$ $b\in \dot{\Lambda}_{\beta}(\Rn)$.

$2.~$ $[b,M_\alpha]$ is bounded from $L^{p}(\Rn)$ to $L^{q}(\Rn)$.

$3.~$ There exists a constant $C > 0$ such that
\begin{equation*}
\sup_{B} \frac{1}{|B|^{\beta/n}}\left(\frac{1}{|B|}\int_{B}|b(x)-|B|^{-\a/n}M_{\a,B}(b)(x)|^q dx\right)^{1/q}\le C.
\end{equation*}
\end{corollary}

\begin{remark}
For $\a=0$, Corollary \ref{ZhLpCoxt} was proved in \cite{ZhArxiv}.
\end{remark}

\begin{remark}
From the proof of Theorem \ref{dnndk} one can see that the assumption $b\geq 0$ is not used in $(2)\Rightarrow (3)$ and $(3)\Rightarrow (1)$. This means (2) and (3) are sufficient conditions for $b\in \dot{\Lambda}_{\beta}(\Rn)$. But we don't know if (2) and (3) are necessary for $b\in \dot{\Lambda}_{\beta}(\Rn)$.
\end{remark}
Indeed, we have obtained the following result.
\begin{corollary}\label{corarx4.1}
Let $0<\beta<1$, $0\le \a<n$, $0< \a+\b<n$ and $b$ be a locally integrable function. Suppose that $\Phi, \Psi$ be Young functions, $\Phi\in\mathcal{Y}\cap \nabla_2$ and $\Psi^{-1}(t) \thickapprox \Phi^{-1}(t)t^{-\frac{\a+\b}{n}}$. If one of the following statements is true, then $b\in \dot{\Lambda}_{\beta}(\Rn)$:

$1.~$ $[b,M_{\a}]$ is bounded from $L^{\Phi}(\Rn)$ to $L^{\Psi}(\Rn)$.

$2.~$ There exists a constant $C > 0$ such that
\begin{equation*}
\sup_{B} |B|^{-\beta/n}\Psi^{-1}\big(|B|^{-1}\big)\|b(\cdot)-|B|^{-\a/n}M_{\a,B}(b)(\cdot)\|_{L^{\Psi}(B)}\le C.
\end{equation*}
\end{corollary}

\begin{theorem}\label{ungo}
Let $b\geq 0$ be a locally integrable function, $0<\beta<1$, $0\le \a<n$, $0< \a+\b<n$ and $b\in \dot{\Lambda}_{\beta}(\Rn)$. Suppose that $\Phi, \Psi$ be Young functions, $\Phi\in\mathcal{Y}$ and condition \eqref{adFrCharOrl1M} holds. Then $[b,M_{\a}]$ is bounded from $L^{\Phi}(\Rn)$ to $WL^{\Psi}(\Rn)$.
\end{theorem}
\begin{proof}
Obviously, it follows from \eqref{3.1gogmus} and Theorem \ref{CommFrMaxCharOr01W}.
\end{proof}

If we take $\Phi(t)=t^{p}$ and $\Psi(t)=t^{q}$ with $1\le p<\i$ and $1\le q \le\i$ at Theorem \ref{ungo}, we have the following result.

\begin{corollary}\label{ZhLpCoWxt}
Let $b\geq 0$ be a locally integrable function, $0<\beta<1$, $0\le \a<n$, $0< \a+\b<n$, $b\in \dot{\Lambda}_{\beta}(\Rn)$, $1\le p<q\le\i$ and $\frac{1}{p}-\frac{1}{q}=\frac{\a+\b}{n}$. Then $[b,M_{\a}]$ is bounded from $L^{p}(\Rn)$ to $WL^{q}(\Rn)$.
\end{corollary}

\begin{remark}
For $\a=0$, Corollary \ref{ZhLpCoWxt} was proved in \cite{ZhArxiv}.
\end{remark}


\begin{thebibliography}{99}
\bibitem{AgcGogKMus} Agcayazi, M., Gogatishvili, A., Koca, K., Mustafayev, R.:  A note on maximal commutators and commutators of maximal functions. J. Math. Soc. Japan \textbf{67}, 581-593 (2015)

\bibitem{CRW} Coifman, R.R., Rochberg, R., Weiss, G.: Factorization theorems for Hardy spaces in several variables. Ann. of Math. \textbf{103}, 611-635 (1976)

\bibitem{Cianchi1999} Cianchi, A.: Strong and weak type inequalities for some classical operators in Orlicz spaces. J. London Math. Soc. \textbf{60}, 247-286 (1999)

\bibitem{DeVSharp} DeVore, R.A., Sharpley, R.C.: Maximal functions measuring smoothness. Mem. Amer. Math. Soc. \textbf{47}(293), viii+115 pp (1984)

\bibitem{GaRa} Gala, S., Ragusa, M.A.: Logarithmically improved regularity criteria for supercritical quasi-geostrophic equations in Orlicz-Morrey spaces. Electron. J. Differential Equations \textbf{137}, 1-7 (2016)

\bibitem{GaRaSawTan} Gala, S., Ragusa, M.A., Sawano, Y., Tanaka, H.: Uniqueness criterion of weak solutions for the dissipative quasi-geostrophic equations in Orlicz-Morrey spaces. Appl. Anal. \textbf{93},  356-368 (2014)

\bibitem{Hedberg} Hedberg, L.I.: On certain convolution inequalities. Proc. Amer. Math. Soc. \textbf{36}, 505-510 (1972)

\bibitem{S.Janson} Janson, S.: Mean oscillation and commutators of singular integral operators. Ark. Mat. \textbf{16}, 263-270 (1978)

\bibitem{Kita1} Kita, H.: On maximal functions in Orlicz spaces. Proc. Amer. Math. Soc. \textbf{124}, 3019-3025 (1996)

\bibitem{KokKrbec} Kokilashvili, V., Krbec, M.M.: Weighted Inequalities in Lorentz and Orlicz Spaces. World Scientific, Singapore, (1991)

\bibitem{Palu} Paluszy\'{n}ski, M.: Characterization of the Besov spaces via the commutator operator of Coifman, Rochberg and Weiss. Indiana Univ. Math. J. \textbf{44}, 1-17 (1995)

\bibitem{RaoRen} Rao, M.M., Ren, Z.D.: Theory of Orlicz Spaces. M. Dekker, Inc., New York, (1991)

\bibitem{ZhLWu} Zhang, P., Wu, J.L.: Commutators of the fractional maximal functions. Acta Math. Sinica (Chin. Ser.) \textbf{52}, 1235-1238 (2009)

\bibitem{ZhWu} Zhang, P., Wu, J.L.: Commutators of the fractional maximal function on variable exponent Lebesgue spaces. Czechoslovak Math. J. \textbf{64}, 183-197 (2014)

\bibitem{ZhArxiv} Zhang, P.: Characterization of Lipschitz spaces via commutators of the Hardy-Littlewood maximal function. C. R. Math. Acad. Sci. Paris \textbf{355}, 336-344 (2017)
\end{thebibliography}
\end{document}